\documentclass[reqno]{amsart}

\usepackage{amsmath}
\usepackage{amssymb}

\newtheorem{Theorem}{Theorem}[section]

\newtheorem{Corollary}[Theorem]{Corollary}

\theoremstyle{definition}
\newtheorem{Definition}[Theorem]{Definition}
\newtheorem{Remark}[Theorem]{Remark}

\numberwithin{equation}{section}

\newcommand\thref[1]{Theorem \ref{#1}}

\newcommand\coref[1]{Corollary \ref{#1}}
\newcommand\deref[1]{Definition \ref{#1}}

\newcommand\reref[1]{Remark \ref{#1}}

\newcommand\seref[1]{Section \ref{#1}}

\newcommand{\Nset}{\mathbb N}
\newcommand{\Cset}{\mathbb C}
\newcommand{\Rset}{\mathbb R}
\newcommand{\Mset}{{\mathbb M}_d}

\newcommand{\de}{\delta}
\newcommand{\be}{\beta}

\newcommand{\diag}{\mathrm{diag}}
\newcommand{\Id}{\mathrm{Id}}

\newcommand{\ct}{\tilde{c}}
\newcommand{\Ct}{\tilde{C}}

\newcommand{\fMd}{\mathfrak{M}_d}
\newcommand{\fb}{\mathfrak{b}}
\newcommand{\fm}{\mathfrak{m}}
\newcommand{\fmt}{\tilde{\mathfrak{m}}}
\newcommand{\Gxz}{G(x,z;\fm,\be)}
\newcommand{\Gxw}{G(x,w;\fm,\be)}
\newcommand{\Pnx}{P_n(x;\fm,\be)}
\newcommand{\Pmx}{P_m(x;\fm,\be)}

\newcommand{\Ex}[1]{E_{x_{#1}}}
\newcommand{\En}[1]{E_{n_{#1}}}

\newcommand{\cLx}[1]{\mathcal{L}^{x}_{#1}}
\newcommand{\cLn}[1]{\mathcal{L}^{n}_{#1}}

\begin{document}

\title[Meixner polynomials in several variables]{Meixner polynomials in several variables satisfying bispectral difference equations}

\author[P.~Iliev]{Plamen~Iliev}

\address{School of Mathematics, Georgia Institute of Technology,
Atlanta, GA 30332--0160, USA}
\email{iliev@math.gatech.edu}
\thanks{The author was supported in part by NSF grant DMS-0901092.}

\date{February 14, 2012}

\maketitle

\begin{abstract}
We construct a set $\fMd$ whose points parametrize families of Meixner polynomials in $d$ variables. There is a natural bispectral involution $\fb$ on $\fMd$ which corresponds to a symmetry between the variables and the degree indices of the polynomials. We define two sets of $d$ commuting partial difference operators diagonalized by the polynomials. One of the sets consists of difference operators acting on the variables of the polynomials and the other one on their degree indices, thus proving their bispectrality. The two sets of partial difference operators are naturally connected via the involution $\fb$.
\end{abstract}

\section{Introduction}\label{se1}

One of the reasons why the classical orthogonal polynomials $p_n(x)$ of a discrete variable $x$ appear in numerous applications is the fact they are eigenfunctions of a second-order difference operator acting on $x$. A detailed account of their characteristic properties and application can be found in \cite{NSU}.

The last few years there have been advances in different aspects of the classification and the construction of orthogonal polynomials in several variables which are eigenfunctions of partial difference operators, see for instance \cite{GI,GR1,I1,I2,IT,IX}. The difficulty in the multivariate case is that the polynomials are no longer uniquely determined by the orthogonality measure (up to a multiplicative constant), and therefore their spectral properties depend on the way we apply the Gram-Schmidt process, see \cite{DX} for the general theory.

Following the pioneering work of Duistermaat and Gr\"unbaum \cite{DG}, we say that a family $\{P_{n_1,\dots,n_d}(x_1,\dots,x_d)\}$ of polynomials in $d$ variables $x_1,\dots,x_d$ with degree indices $n_1,\dots,n_d$ solves a discrete-discrete bispectral problem, if there exist commuting partial difference operators $\cLx{1},\dots,\cLx{d}$ acting on the variables $x_1,\dots,x_d$ with coefficients independent of $n_1,\dots,n_d$ and commuting partial difference operators $\cLn{1},\dots,\cLn{d}$ acting on the degree indices $n_1,\dots,n_d$ with coefficients independent of $x_1,\dots,x_d$ such that 
\begin{subequations} \label{1.1}
\begin{align}
&\cLx{j}P_{n_1,\dots,n_d}(x_1,\dots,x_d)=\lambda_j(n_1,\dots,n_d) P_{n_1,\dots,n_d}(x_1,\dots,x_d) \label{1.1a}\\
&\cLn{j}P_{n_1,\dots,n_d}(x_1,\dots,x_d)=\mu_j(x_1,\dots,x_d)P_{n_1,\dots,n_d}(x_1,\dots,x_d), \label{1.1b}
\end{align}
\end{subequations}
for $j=1,2,\dots,d$, where $\lambda_j(n_1,\dots,n_d)$ are independent of $x_1,\dots,x_d$ and $\mu_j(x_1,\dots,x_d)$ are independent of $n_1,\dots,n_d$. Within the context of orthogonal polynomials one is often interested in the commuting operators corresponding to multiplications by the independent variables $x_j$ and therefore $\mu_j(x_1,\dots,x_d)=x_j$ is one natural choice for the eigenvalues in \eqref{1.1b}. On the other hand, even in the one-variable case, the eigenvalues in \eqref{1.1a} can be quadratic functions (e.g. in the case of Hahn polynomials) and in higher dimensions, they can depend on some or all of the chosen degree indices $n_1,\dots,n_d$, see for instance the families discussed in \cite{GI,I1}.

Recently \cite{I2} we proved that all families of multivariate 
Krawtchouk polynomials defined by Griffiths \cite{G} are bispectral. We used a Lie-theoretic approach which naturally led to the construction of the partial difference operators in equations \eqref{1.1}. Moreover, we showed how other known families of bispectral multivariate Krawtchouk polynomials fit within the above theory. The construction was a natural extension of an earlier joint work with Terwilliger \cite{IT} for the bivariate polynomials defined in \cite{HR}.  We note that different proofs of the orthogonality properties of the multivariate Krawtchouk polynomials as well as interesting probabilistic applications and connections to Gelfand pairs and character algebras were discussed in the papers \cite{GR2,M1,M2,MT}.

In the present paper we give a similar parametrization of bispectral multivariate Meixner polynomials and the corresponding partial difference operators. The main result concerning the bispectrality of the multivariate Meixner polynomials constructed here is analogous to one in the Krawtchouk case \cite{I2}, but the proof is different.

The paper is organized as follows. In the next section, we define a set $\fMd$ and for every $\fm\in\fMd$ and $\be\in\Cset\setminus(-\Nset_0)$ we construct multivariate Meixner polynomials $\{P_{n_1,\dots,n_d}(x_1,\dots,x_d;\fm,\be):n_j\in\Nset_0\}$, 
mutually orthogonal on $\Nset_0^d$ with respect to a weight function of the form
$$W(x)=(\be)_{\sum_{j=1}^{d}x_j}\prod_{j=1}^{d}\frac{c_j^{x_j}}{x_j!}.$$ 
The construction is based on a generating function and is similar to one used by Griffiths \cite{G} in the case of the Krawtchouk 
polynomials. In \seref{se3} we give a hypergeometric representation for the multivariate Meixner polynomials which yields the duality between the variables and the degree indices of the polynomials. The key ingredient is to ``expand'' the generating function appropriately and to regroup the terms by changing the summation indices. The computation is essentially the same as the one in the Krawtchouk case which can found in the work of Mizukawa and Tanaka \cite{MT}, except that the expansion of the generating series now is infinite. In \seref{se4} we construct the commuting partial difference operators $\{\cLx{j}\}_{j=1}^{d}$, $\{\cLn{j}\}_{j=1}^{d}$  and we prove that they are diagonalized by the polynomials  $\{P_{n_1,\dots,n_d}(x_1,\dots,x_d;\fm,\be):n_j\in\Nset_0\}$ by using the generating function and the duality. Finally, we list two explicit families of Meixner polynomials in arbitrary dimension in \seref{se5}. In the first example, we explain how we need to pick the point $\fm\in\fMd$ in order to obtain the multivariate Meixner polynomials defined in \cite{IX}. In the second example, we construct a point $\fm\in\fMd$ which depends on one free parameter.  The corresponding Meixner polynomials can be thought of as analogs of the the multivariate Krawtchouk polynomials used in \cite{DS}.

\section{Multivariate Meixner polynomials}\label{se2}

\subsection{Basic notations and definitions}\label{ss2.1}
Throughout the paper we shall use standard multi-index notations. For instance, 
if $x=(x_1,x_2,\dots,x_d)\in\Cset^d$ and $n=(n_1,n_2,\dots,n_d)\in\Nset_0^d$ then 
\begin{equation*}
x^n=x_1^{n_1}x_2^{n_2}\cdots x_d^{n_d},\qquad n!=n_1!n_2!\cdots n_d!,
\end{equation*}
and $|x|=x_1+x_2+\cdots+x_d$.

First we introduce the set $\fMd$ whose points parametrize families of multivariate Meixner polynomials. In most applications, the parameters are real, but since the partial difference equations studied here have natural extensions for complex numbers, we work below with $\Cset$ rather than $\Rset$.

\begin{Definition}\label{de2.1} 
Let  $\fMd$ denote the set of 4-tuples $(c_0,C,\Ct,U)$,
where $c_0$ is a nonzero complex number 
and $C,\Ct, U$ are $(d+1)\times (d+1)$ matrices with complex entries
satisfying the following conditions:
\begin{itemize}
\item[{(i)}] $C=\diag(1,-c_1,-c_2,\dots,-c_{d})$ and 
$\Ct=\diag(1,-\ct_1,-\ct_2\dots,-\ct_{d})$ are diagonal;
\item[{(ii)}] $U=(u_{i,j})_{0\leq i,j\leq d}$ is such that 
$u_{0,j}=u_{j,0}=1$ for all $j=0,1,\dots,d$, i.e. 
\begin{equation}\label{2.1}
U=\left(\begin{matrix}1 & 1 & 1 & \dots &1\\ 
1 & u_{1,1} &u_{1,2} &\dots & u_{1,d}\\
 \vdots & \\
 1 & u_{d,1} & u_{d,2}  &\dots &u_{d,d}
\end{matrix}\right);
\end{equation}
\item[{(iii)}]  The following matrix equation holds
\begin{equation}\label{2.2}
U^{t}CU\Ct =c_0I_{d+1}, 
\end{equation}
where  $I_{d+1}$ denotes the identity $(d+1)\times (d+1)$ matrix.
\end{itemize}
\end{Definition}
Note that $c_0$ is a nonzero number and therefore \eqref{2.2} implies that the matrices $C$, $\Ct$ and $U$ in \deref{de2.1} are invertible. We shall denote by $c$ and $\ct$ the $d$-dimensional vectors 
$$c=(c_1,c_2,\dots,c_d), \qquad \ct=(\ct_1,\ct_2,\dots,\ct_d).$$ 

Comparing the $(i,j)$ entries on both sides of equation \eqref{2.2} for $j=0$ and $j>0$ we obtain the following identities:
\begin{subequations}\label{2.3}
\begin{align}
&\sum_{k=1}^{d}c_ku_{k,i}=1-c_0\de_{i,0}\label{2.3a}\\
&\sum_{k=1}^{d}c_ku_{k,i}u_{k,j}=1+\frac{c_0}{\ct_j}\de_{i,j}.\label{2.3b}
\end{align}
\end{subequations}

\begin{Remark}\label{re2.2}
We shall use later the involution $\fb$ on $\fMd$ defined by 
\begin{equation}\label{2.4}
\fb:\fm=(c_0,C,\Ct,U)\rightarrow \fmt=(c_0,\Ct,C,U^t).
\end{equation}
In particular, using \eqref{2.3a} with $i=0$ and $\fb$ we see that 
\begin{equation}\label{2.5}
c_0=1-|c|= 1-|\ct|.
\end{equation}
\end{Remark}

\begin{Remark}
If we start with nonzero numbers $c_1,c_2,\dots,c_d$ such that $c_0=1-|c|\neq 0$ we can define a point $\fm=(c_0,C,\Ct,U)\in\fMd$ as follows. First, we construct a basis $v_0,v_1,\dots,v_d$ of column vectors in $\Cset^{d+1}$ mutually orthogonal with respect to the bilinear form $\langle w_1,w_2\rangle=w_1^t C w_2$ such that $v_0=(1,1,\dots,1)^t$, the $0$-th coordinate of $v_j$ is $1$  and $\langle v_j, v_j\rangle\neq0$ for $j=1,2,\dots,d$. Then the matrix $U$ with columns $v_0,v_1,\dots,v_d$ is of the form given in \eqref{2.1} and the diagonal matrix $\Ct$ is uniquely determined by \eqref{2.2}. 
If $d=1$, then $v_1=(1,1/c_1)^t$ is uniquely determined from $c_1$. However, when $d>1$ we have $d(d-1)/2$ degrees of freedom in choosing $U$.
\end{Remark}

In the rest of the paper we use $\{z_j\}_{j=1}^{d}$ and $\{w_j\}_{j=1}^{d}$ to denote formal complex variables which are sufficiently close to $0$. In all exponents below we fix the principal branch of the logarithmic function.

For every $\fm\in\fMd$ and $\be\in\Cset$ we consider the function of $x=(x_1,\dots,x_d)$, 
$z=(z_1,\dots,z_d)$ 
\begin{equation}\label{2.6}
\Gxz=\left(1-|z|\right)^{-\be-|x|}\prod_{i=1}^{d}
\left(1-\sum_{j=1}^{d}u_{i,j}z_j\right)^{x_i}.
\end{equation}
\begin{Definition}\label{de2.3}
For $\fm\in\fMd$ and $\be\in\Cset\setminus(-\Nset_0)$ we define multivariate 
Meixner polynomials $\{\Pnx:n\in\Nset_0^d\}$ by expanding  $\Gxz$ in a neighborhood of $z=0$ as follows
\begin{equation}\label{2.7}
\Gxz=\sum_{n\in\Nset_0^d}\frac{(\be)_{|n|}}{n!}\Pnx z^n.
\end{equation}
\end{Definition}

\subsection{Orthogonality relations}\label{ss2.2}
We show next that the polynomials $\{\Pnx:n\in\Nset_0^d\}$ defined above are mutually orthogonal with respect to the weight $(\be)_{|x|}c^x/x!$ on $\Nset_0^d$. More precisely, using the notations introduced so far, the following theorem holds.
\begin{Theorem}\label{th2.4}
Suppose that $|c_1|+|c_2|+\cdots+|c_d|<1$. Then for $n,m\in\Nset_0^d$ we have 
\begin{equation}\label{2.8}
\sum_{x\in\Nset_0^d}\Pnx \Pmx \frac{(\be)_{|x|}}{x!}c^x
=\frac{c_0^{-\be}n!}{(\be)_{|n|}\,\ct^{n}}\de_{n,m}.
\end{equation}
\end{Theorem}
\begin{proof}
Using the formula
\begin{equation}\label{2.9}
\left(1-|z|\right)^{-\gamma}=\sum_{k\in\Nset_0^d} \frac{(\gamma)_{|k|}}{k!}z^k,
\end{equation}
we find 
\begin{equation}\label{2.10}
\begin{split}
&\sum_{x\in\Nset_0^d}\Gxz \Gxw \frac{(\be)_{|x|}}{x!}c^x\\
&\quad=\left[(1-|z|)(1-|w|)-\sum_{i=1}^{d}c_i\left(1-\sum_{j=1}^{d}u_{i,j}z_j\right)\left(1-\sum_{s=1}^{d}u_{i,s}w_s\right)\right]^{-\be}.
\end{split}
\end{equation}
If we set $Z=(-1,z_1,\dots,z_d)^{t}$ and $W=(-1,w_1,\dots,w_d)^{t}$, then the expression in the big parentheses on the right-hand side 
above is equal to $(UW)^{t}C(UZ)$. Using \eqref{2.2} we find
\begin{equation*}
(UW)^{t}C(UZ)=c_0W^{t}\Ct^{-1}Z=c_0\left[1-\sum_{j=1}^{d}\frac{z_jw_j}{\ct_j}\right].
\end{equation*}
Thus, \eqref{2.10} gives
$$\sum_{x\in\Nset_0^d}\Gxz \Gxw \frac{(\be)_{|x|}}{x!}c^x
=c_0^{-\be}\left[1-\sum_{j=1}^{d}\frac{z_jw_j}{\ct_j}\right]^{-\be}.$$
The proof now follows by expanding the right side of the last equation using \eqref{2.9} and comparing the coefficients of $z^nw^m$ on both sides of the resulting identity.
\end{proof}

\section{Hypergeometric representation}\label{se3}
We denote by $\Mset$ the set of all $d\times d$ matrices with entries in $\Nset_0$.

\begin{Theorem}\label{th3.1}
We have
\begin{equation}\label{3.1}
\Pnx=\sum_{A=(a_{i,j})\in\Mset}\frac{\prod_{j=1}^{d}(-n_j)_{\sum_{i=1}^{d}a_{i,j}}\prod_{i=1}^{d}(-x_i)_{\sum_{j=1}^{d}a_{i,j}}}{(\be)_{\sum_{i,j=1}^{d}a_{i,j}}}\prod_{i,j=1}^{d}\frac{(1-u_{i,j})^{a_{i,j}}}{a_{i,j}!}.
\end{equation}
\end{Theorem}
Note that $(-n_j)_{\sum_{i=1}^{d}a_{i,j}}=0$ when $\sum_{i=1}^{d}a_{i,j}>n_j$. Therefore the sum in \eqref{3.1} is finite, using only matrices $A$ with nonnegative integer entries satisfying $\sum_{i=1}^{d}a_{i,j}\leq n_j$ for $j=1,\dots,d$. 
\begin{proof}[Proof of \thref{th3.1}]
For every $i\in\{1,2,\dots, d\}$ we have
\begin{equation*}
1-\sum_{j=1}^{d}u_{i,j}z_j=(1-|z|)\left(1+\sum_{j=1}^{d}\frac{(1-u_{i,j})z_j}{1-|z|}\right).
\end{equation*}
Taking the $x_i$-th power of the above equation and expanding the second term using \eqref{2.9} we obtain
\begin{equation*}
\left(1-\sum_{j=1}^{d}u_{i,j}z_j\right)^{x_i}=(1-|z|)^{x_i}\sum_{k_i\in\Nset_0^d}\frac{(-1)^{|k_i|}(-x_i)_{|k_i|}}{k_i!}\frac{\eta_i^{k_i}z^{k_i}}{(1-|z|)^{|k_i|}},
\end{equation*}
where we have set $\eta_i=(1-u_{i,1},1-u_{i,2},\dots,1-u_{i,d})$. Substituting the last formula into the right-hand side of \eqref{2.6} we obtain
\begin{equation*}
\begin{split}
\Gxz=&\sum_{k_1,k_2,\dots,k_d\in\Nset_0^d}(1-|z|)^{-\be-\sum_{i=1}^{d}|k_i|}(-1)^{\sum_{i=1}^{d}|k_i|}\\
&\quad\times \left(\prod_{i=1}^{d}\frac{(-x_i)_{|k_i|}}{k_i!}\eta_i^{k_i}\right)z^{\sum_{i=1}^{d}k_i}.
\end{split}
\end{equation*}
Expanding also $(1-|z|)^{-\be-\sum_{i=1}^{d}|k_i|}$ we get
\begin{equation}\label{3.2}
\begin{split}
\Gxz=&\sum_{l,k_1,k_2,\dots,k_d\in\Nset_0^d}\frac{(-1)^{\sum_{i=1}^{d}|k_i|}\left(\be+\sum_{i=1}^{d}|k_i|\right)_{|l|}}{l!}\\
&\quad\times \left(\prod_{i=1}^{d}\frac{(-x_i)_{|k_i|}}{k_i!}\eta_i^{k_i}\right)z^{l+\sum_{i=1}^{d}k_i}.
\end{split}
\end{equation}
Let us denote by $a_{i,j}$ the entries of the vector $k_i$, and by $l_j$ the entries of $l$, i.e. $k_i=(a_{i,1},a_{i,2},\dots,a_{i,d})$ and 
$l=(l_1,l_2,\dots,l_d)$. Then 
\begin{equation}\label{3.3}
\frac{(-1)^{\sum_{i=1}^{d}|k_i|}\left(\be+\sum_{i=1}^{d}|k_i|\right)_{|l|}}{l!}
=\frac{(\be)_{|l|+\sum_{i,j=1}^{d}a_{i,j}}}{(\be)_{\sum_{i,j=1}^{d}a_{i,j}}}\prod_{j=1}^{d}\frac{\left(-l_j-\sum_{i=1}^{d}a_{i,j}\right)_{\sum_{i=1}^{d}a_{i,j}}}{\left(l_j+\sum_{i=1}^{d}a_{i,j}\right)!}.
\end{equation}
The proof now follows by plugging \eqref{3.3} in the first line of \eqref{3.2} and by replacing the sum over $l$ with a sum over $n=(n_1,n_2,\dots,n_d)$ where $n_j=l_j+\sum_{i=1}^{d}a_{i,j}$.
\end{proof}
As an immediate corollary of the hypergeometric representation \eqref{3.1} and the involution $\fb$ on $\fMd$ in \reref{re2.2} we see that the polynomials $\Pnx$ possess a duality between the variables $\{x_j\}$ and the indices $\{n_j\}$.
\begin{Corollary}\label{co3.2}
For $n,x\in\Nset_0^d$ we have
\begin{equation}\label{3.4}
\Pnx=P_x(n,\fmt,\be).
\end{equation} 
\end{Corollary}
\section{Bispectrality}\label{se4}
Let $\{e_1,e_2,\dots,e_d\}$ be the standard basis for $\Cset^{d}$. For every $i\in\{1,2,\dots,d\}$ we denote by $\Ex{i}$ and $\En{i}$ the customary shift operators acting on functions of $x=(x_1,x_2,\dots,x_d)$ and $n=(n_1,n_2,\dots,n_d)$, respectively, as follows
$$\Ex{i}f(x)=f(x+e_i) \quad\text{ and }\quad\En{i}g(n)=g(n+e_i).$$
For fixed $\fm\in\fMd$, $\be\in\Cset$ and $i\in\{1,2,\dots,d\}$ we define the following difference operators
\begin{subequations}\label{4.1}
\begin{align}
\cLx{i}&=\frac{\ct_i}{c_0}\sum_{1\leq k\neq l\leq d}c_ku_{k,i}u_{l,i}x_{l}\left(\Ex{k}\Ex{l}^{-1}-\Id\right)\nonumber\\
&\quad-\frac{\ct_i}{c_0}\sum_{l=1}^{d}u_{l,i}x_l\left(\Ex{l}^{-1}-\Id\right)-\frac{\ct_i}{c_0}\sum_{k=1}^{d}c_ku_{k,i}(\be+|x|)\left(\Ex{k}-\Id\right),\label{4.1a}
\intertext{and}
\cLn{i}&=\frac{c_i}{c_0}\sum_{1\leq k\neq l\leq d}\ct_ku_{i,k}u_{i,l}n_{l}\left(\En{k}\En{l}^{-1}-\Id\right)\nonumber\\
&\quad-\frac{c_i}{c_0}\sum_{l=1}^{d}u_{i,l}n_l\left(\En{l}^{-1}-\Id\right)-\frac{c_i}{c_0}\sum_{k=1}^{d}\ct_ku_{i,k}(\be+|n|)(\En{k}-\Id),\label{4.1b}
\end{align}
\end{subequations}
where $\Id$ denotes the identity operator. We show below that these operators are diagonalized by the polynomials $\Pnx$, thus providing solutions to the bispectral problem. Note that the backward shift operators $\En{l}^{-1}$ in \eqref{4.1b} are multiplied by $n_l$, and therefore these expressions will vanish when $n_l=0$. Thus $\cLn{i}$ is a well-defined operator acting on functions of $n\in\Nset_0^d$. Similarly, in view of the orthogonality \eqref{2.8}, it is natural to consider the polynomials for $x\in\Nset_0^d$ and the operators $\cLx{i}$ will involve evaluations of $\Pnx$ only at points $x$ with nonnegative integer coordinates.

\begin{Theorem}\label{th4.1}
For every $i\in\{1,2,\dots,d\}$ the following spectral equations hold
\begin{subequations}\label{4.2}
\begin{align}
\cLx{i}\Pnx&=n_i\Pnx,\label{4.2a} \\
\cLn{i}\Pnx&=x_i\Pnx.\label{4.2b}
\end{align}
\end{subequations}
\end{Theorem}
\begin{proof}
Fix $i\in\{1,2,\dots,d\}$ and let us denote 
\begin{equation*}
\cLx{i,l}=\Id-\Ex{l}^{-1}+\sum_{k\in\{1,\dots,d\}\setminus\{l\}}c_ku_{k,i}\left(\Ex{k}\Ex{l}^{-1}-\Id\right)\text{ for  every }l=1,2,\dots,d
\end{equation*}
and 
\begin{equation*}
\cLx{i,+}=-\frac{\ct_i}{c_0}\sum_{k=1}^{d}c_ku_{k,i}(\be+|x|)\left(\Ex{k}-\Id\right).
\end{equation*}
Then 
\begin{equation}\label{4.3}
\cLx{i}=\frac{\ct_i}{c_0}\sum_{l=1}^{d}u_{l,i}x_l\cLx{i,l}+\cLx{i,+}.
\end{equation}
Using the definition \eqref{2.6} of $\Gxz$ we find
\begin{align*}
&\frac{1}{\Gxz}\sum_{k\in\{1,\dots,d\}\setminus\{l\}}c_ku_{k,i}\left(\Ex{k}\Ex{l}^{-1}-\Id\right)\Gxz\\
&\quad=\frac{1}{\left(1-\sum_{j=1}^{d}u_{l,j}z_j\right)}\sum_{k\in\{1,\dots,d\}\setminus\{l\}}\,\sum_{j=1}^{d}(c_ku_{k,i}u_{l,j}-c_ku_{k,i}u_{k,j})z_j\\
\intertext{(interchanging the sums)}
&\quad= \frac{1}{\left(1-\sum_{j=1}^{d}u_{l,j}z_j\right)}\sum_{j=1}^{d}\sum_{k=1}^{d}(c_ku_{k,i}u_{l,j}-c_ku_{k,i}u_{k,j})z_j\\
\intertext{(using equations \eqref{2.3})}
&=\quad \frac{1}{\left(1-\sum_{j=1}^{d}u_{l,j}z_j\right)}\sum_{j=1}^{d}\left(u_{l,j}-\frac{c_0}{\ct_j}\de_{i,j}-1\right)z_j\\
&=\quad \frac{1}{\left(1-\sum_{j=1}^{d}u_{l,j}z_j\right)}\left(\sum_{j=1}^{d}(u_{l,j}-1)z_j-\frac{c_0}{\ct_i}z_i\right).
\end{align*}
From the last relation, the definition of $\cLx{i,l}$ 
and 
\begin{equation*}
\frac{1}{\Gxz}(\Id-\Ex{l}^{-1})\,\Gxz=\frac{\sum_{j=1}^{d}(1-u_{l,j})z_j}{1-\sum_{j=1}^{d}u_{l,j}z_j}
\end{equation*}
it follows that 
\begin{equation}\label{4.4}
\frac{1}{\Gxz}\cLx{i,l}\,\Gxz=-\frac{c_0z_i}{\ct_i\left(1-\sum_{j=1}^{d}u_{l,j}z_j\right)}.
\end{equation}
A similar computation shows that
\begin{equation}\label{4.5}
\frac{1}{\Gxz}\cLx{i,+}\,\Gxz=\frac{(\be+|x|)z_i}{1-|z|}.
\end{equation}
Using equations \eqref{4.3}, \eqref{4.4} and \eqref{4.5} we see that 
\begin{equation}\label{4.6}
\cLx{i}\Gxz=z_i\frac{\partial }{\partial z_i}\Gxz,
\end{equation}
which combined with \eqref{2.7} completes the proof of \eqref{4.2a}. The proof of \eqref{4.2b} follows from the duality established in \coref{co3.2}. 
Indeed, if $x\in\Nset_0^{d}$, then $\Pnx=P_x(n,\fmt,\be)$ and equation \eqref{4.2b} follows from \eqref{4.2a}. Moreover, for fixed $n\in\Nset_0^{d}$ both sides of \eqref{4.2b} are polynomials in $x$ of total degree at most $|n|+1$, hence if the equality holds for all $x\in\Nset_0^{d}$ it will be true for arbitrary $x\in\Cset^{d}$.
\end{proof}

\section{Some examples}\label{se5}
In this section we illustrate how specific points in $\fMd$ lead to families of multivariate Meixner polynomials or analogs of multivariate Krawtchouk polynomials which have already appeared in the literature in different applications.

\subsection{}\label{ss5.1}
Let us fix  nonzero complex numbers $c_1,\dots,c_{d}$ and let $c_0=1-|c|$. We define $\ct_1,\ct_2,\dots,\ct_{d}$ in terms of  $c_1,\dots,c_{d}$  as follows:
\begin{equation}\label{5.1}
\ct_{k}=\frac{c_kc_0}{(1-\sum_{j=k+1}^{d}c_j)(1-\sum_{j=k}^{d}c_j)}\text{ for }k=1,2,\dots,d.
\end{equation}
We shall assume that $\{c_j\}$ are such that the denominators in \eqref{5.1} do not vanish. Next we define a $(d+1)\times(d+1)$ matrix $U=(u_{i,j})$ with entries
\begin{subequations}\label{5.2}
\begin{align}
&u_{i,j}=\delta_{0,i},  &\text{ when }i<j,\label{5.2a}\\
&u_{i,j}=1,                 &\text{ when }i>j,\label{5.2b}\\
&u_{0,0}=1,\label{5.2c}\\
&u_{i,i}=-\frac{1-\sum_{k=i+1}^{d}c_k}{c_{i}}, &\text{ for }i=1,2,\dots,d.\label{5.2d}
\end{align}
\end{subequations}
Thus, $U$ is a matrix of the form \eqref{2.1} where the remaining entries are 0's and 1's above and below the diagonal respectively, and the diagonal entries are given in \eqref{5.2d}. One can check that with the above notations equation \eqref{2.2} holds and therefore we obtain a point $\fm\in\fMd$ which depends on the free parameters $c_1,\dots,c_d$. In particular, note that if $c_i\in(0,1)$ for all $i\in\{0,1,\dots,d\}$ then $\ct_i\in(0,1)$ for all $i\in\{1,\dots,d\}$.
Up to a permutation of the variables and the parameters, this choice leads to the multivariate Meixner polynomials defined in \cite{IX}.

\subsection{}\label{ss5.2}
Fix now $q\notin\{0, 1\}$ and define
\begin{equation}\label{5.3}
c_{k}=\ct_{k}=(1-q)q^{k-1}\text{ for }k=1,2,\dots,d.
\end{equation}
Then $c_0=1-|c|=q^{d}$. We define also a $(d+1)\times(d+1)$ matrix $U=(u_{i,j})_{0\leq i,j\leq d}$ with entries
\begin{subequations}\label{5.4}
\begin{align}
&u_{i,j}=1,  &\text{ when }i+j\leq d,\label{5.4a}\\
&u_{i,j}=\frac{1}{1-q},                 &\text{ when }i+j=d+1,\label{5.4b}\\
&u_{i,j}=0,&\text{ when }i+j>d+1.\label{5.4c}
\end{align}
\end{subequations}
With the above notations, we see that \eqref{2.2} holds leading to a point $\fm\in\fMd$ which depends on the free parameter $q$. Moreover, if $q\in(0,1)$ then $c_i=\ct_i\in(0,1)$ for all $i\in\{0,1,\dots,d\}$. These Meixner polynomials can be thought of as analogs of the multivariate Krawtchouk polynomials used in \cite{DS}.

\section*{Acknowledgments} The author would like to thank a referee for suggestions to improve an earlier version of this paper.


\begin{thebibliography}{xx}

\bibitem{DS} S.~T.~Dougherty and M.~M.~Skriganov, {\em MacWilliams duality and the Rosenbloom-Tsfasman metric}, 
Mosc. Math. J. 2 (2002), no. 1, 81--97.

\bibitem{DG} J.~J.~Duistermaat and F.~A.~Gr\"unbaum, {\em Differential equations in the spectral parameter},  
Comm. Math. Phys. 103 (1986), no. 2, 177--240.

\bibitem{DX}  C.~F.~Dunkl and Y.~Xu, {\em Orthogonal polynomials of 
several variables}, Encyclopedia of Mathematics and its Applications 81, 
Cambridge University Press, Cambridge, 2001.

\bibitem{GI} J.~Geronimo and P.~Iliev, {\em Bispectrality of multivariable 
Racah-Wilson polynomials},  Constr. Approx. 31 (2010), no. 3, 417--457 (arXiv:0705.1469).

\bibitem{G} R.~C.~Griffiths, {\em Orthogonal polynomials on the multinomial distribution}, Austral. J. Statist. 13 (1971), no. 1, 27--35. 

\bibitem{GR1} F.~A.~Gr\"unbaum and M.~Rahman, {\em On a family of 2-variable orthogonal Krawtchouk polynomials},  SIGMA Symmetry Integrability Geom. Methods Appl.  6  (2010), Paper 090, 12 pp (arXiv:1007.4327).

\bibitem{GR2} F.~A.~Gr\"unbaum and M.~Rahman, {\em A system of multivariable Krawtchouk polynomials
and a probabilistic application}, SIGMA Symmetry Integrability Geom. Methods Appl.  7 (2011), Paper 119, 17pp (arXiv:1106.1835).

\bibitem{HR} M.~R.~Hoare and M.~Rahman, {\em A probabilistic origin for a new class of bivariate polynomials}, SIGMA Symmetry Integrability Geom. Methods Appl.  4  (2008), Paper 089, 18 pp (arXiv:0812.3879).

\bibitem{I1} P. Iliev, {\em Bispectral commuting difference operators for multivariable Askey-Wilson polynomials},  Trans. Amer. Math. Soc. 363  (2011), no. 3, 1577--1598 (arXiv:0801.4939).

\bibitem{I2} P.~Iliev, {\em A Lie-theoretic interpretation of multivariate hypergeometric polynomials}, to appear in Compos. Math. (arXiv:1101.1683). 

\bibitem{IT}  P.~Iliev and P.~Terwilliger, {\em The Rahman polynomials and the Lie algebra $\mathfrak{sl}_3(\Cset)$}, to appear in Trans. Amer. Math. Soc. (arXiv:1006.5062).

\bibitem{IX} P.~Iliev and Y.~Xu, {\em Discrete orthogonal polynomials and difference equations of several variables}, Adv. Math. 212 (2007), no. 1, 1--36 (arXiv:math/0508039).  

\bibitem{M1} H.~Mizukawa, {\em Zonal spherical functions on the complex reflection groups and $(n+1,m+1)$-hypergeometric functions},  Adv. Math.  184  (2004),  no. 1, 1--17.

\bibitem{M2} H.~Mizukawa, {\em Orthogonality relations for multivariate Krawtchouk polynomials}, SIGMA Symmetry Integrability Geom. Methods Appl. 7  (2011), Paper 017, 5 pp. 

\bibitem{MT} H.~Mizukawa and H.~Tanaka, {\em $(n+1,m+1)$-hypergeometric functions associated to character algebras}, Proc. Amer. Math. Soc. 132 (2004), no. 9, 2613--2618.

\bibitem{NSU}  A.~F.~Nikiforov, S.~K.~Suslov and V.~B.~Uvarov, {\em 
Classical orthogonal polynomials of a discrete variable}, Springer Series in 
Computational Physics, Springer-Verlag, Berlin, 1991.

\end{thebibliography}
\end{document}